\documentclass{amsart}
\usepackage {amssymb}
\usepackage {amsmath}
\usepackage{amsthm}
\usepackage{graphicx}
\usepackage {amscd}
\usepackage {epic}
\usepackage[alphabetic]{amsrefs}
\usepackage{xypic}
\usepackage[colorlinks, citecolor=blue]{hyperref}

\newcommand{\Pic}[0]{\operatorname{Pic}}

\def\O{{\mathcal{O}}}

\def\Z{{\mathbb{Z}}}
\def\N{{\mathbb{N}}}
\def\Q{{\mathbb{Q}}}

\def\C{{\mathbb{C}}}

\def\Pic{{\mathrm{Pic}}}

\def\rk{{\mathrm{rk}}}

\newcommand{\vol}{{\rm vol}}

\newtheorem{thm}{Theorem}[section]

\newtheorem{lem}[thm]{Lemma}
\newtheorem{cor}[thm]{Corollary}

\theoremstyle{definition}
\newtheorem{defn}[thm]{Definition}

\newtheorem{rem}[thm]{Remark}

\newtheorem{defn-thm}[thm]{Definition--Theorem}  
\newtheorem{defn-prop}[thm]{Definition--Proposition}  
\newtheorem{defn-lem}[thm]{Definition--Lemma}  

\theoremstyle{remark}

\begin{document}

\title{Nonvanishing for threefolds in characteristic $p>5$}

\author {Chenyang Xu}
\address {Current: MIT, Cambridge, MA 02139, USA}
\address   {Former: Beijing International Center for Mathematical Research,
       Beijing 100871, China}
\email     {cyxu@math.mit.edu, cyxu@math.pku.edu.cn}

\author {Lei Zhang}
\address {Current: School of Mathematical Science, University of Science and Technology of China, Hefei 230026, P.R.China}
\address {Former: School of Mathematics and Information Sciences, Shaanxi Normal
University, Xi'an 710062, China}
\email {zhlei18@ustc.edu.cn, zhanglei2011@snnu.edu.cn}

\maketitle
\begin{abstract}
We prove the Nonvanishing Theorem for threefolds over an algebraically closed field $k$ of characteristic $p >5$.
\end{abstract}

\tableofcontents

\section{Introduction} \label{section:intro}
Starting from \cite{HX15}, the minimal model program for threefolds in positive characteristic $p>5$ has moved forward quickly. Running a minimal model program in this setting is established by  \cite{HX15, CTX15, Birkar16, BW14, HNT17} in increasing generalities. The aim of this paper is to attack the abundance conjecture for minimal threefolds over characteristic $p>5$. Previously, this result was known for $X$ of general type \cite{Birkar16, Xu15} or satisfying $\dim {\rm Pic}^0(X) >0$ \cite{DW16, Zhang16, Zhang17}. The main result in this paper is the following nonvanishing theorem.
\begin{thm}[Non-vanishing theorem]\label{thm-nonvanishing}
Let $X$ be a terminal projective minimal threefold over an algebraically closed field $k$ of characteristic $p > 5$.  Then $\kappa(X, K_X) \ge 0$, that is, nonvanishing holds for $X$.
\end{thm}

Let us recall the proof of this theorem in characteristic zero which is the main theorem in \cite{Miy88a} (see also \cite[Section 9]{Kol92}). It has a different flavour from the proof of the results for running the minimal model program, namely, it involves
more investigation of higher rank bundles (e.g. the tangent bundle), as opposed to focusing on line
bundles (e.g. the canonical bundle). Two of the main tools used in the proof are the pseudo-effectivity of $c_2$ and Donaldson's
theorem on stable bundles on smooth complex projective surfaces. As well known, these results usually do not hold in in positive characteristics. Luckily, the general understanding of the stability of sheaves in positive characteristics has provided us with enough tools to deal with the new positive characteristic phenomenon.

Let us discuss how to treat the first issue: Let $\rho\colon Y\to X$ be a smooth resolution of singularities, i.e., $Y$ is smooth and $\rho$ is a birational morphism isomorphic over the smooth locus $X^{\rm sm}$. The pseudo-effectivity of $c_2$ means $\rho^*K_X\cdot c_2(Y)\ge 0$, which is a starting point in \cite{Miy88a} (clearly the intersection number does not depend on the choice of the smooth resolutions). As we mentioned, this kind of results may not hold in positive characteristics. Nevertheless, it was understood before that the failure of this type of inequalities will usually imply a special geometric structure on $Y$, e.g., a foliation by rational curves. Using this, we can indeed show the full abundance under the assumption that $\rho^*K_X \cdot c_2(Y) <0$.

\begin{thm}[=Theorem \ref{nc2}]\label{case-nc2}
Let $X$ be a terminal projective minimal 3-fold over an algebraically closed field $k$ of characteristic $p \geq 5$. Let $\rho: Y \to X$ be a smooth resolution of singularities. Assume that $\nu(X, K_X) <3$ and $\rho^*K_X \cdot c_2(Y) <0$. Then  $\kappa(K_X) =1$ and $K_X$ is semiample.
\end{thm}
We want to note that in the proof of Theorem \ref{case-nc2}, we do not need to run any minimal model program. This is the reason why we can treat the ${\rm char}~k=5$ case.

More precisely, to prove Theorem \ref{case-nc2}, we start with the assumption $\rho^*K_X \cdot c_2(Y) <0$, which is equivalent to $\Delta(T_Y) \cdot \rho^*K_X < 0$ as
$$\rho^*K_X \cdot K_Y^2=K_X^3=0$$
since $\rho$ is isomorphic outside isolated points on $X$.
Using Bogomolov inequality
\cite[Theorem 3.2]{Lan04}, we immediately get that the tangent bundle $T_Y$ is not strongly semistable, and in our situation we can show that $T_Y$ is unstable. Therefore, we can use the results in \cite{Lan15} to get a foliation $\mathcal{F}$ on $Y$, and a family of $\rho^*K_X$-trivial rational curves covering $X$, hence the nef dimension of $K_X$ satisfies $n(X, K_X) \leq 2$. Under the assumption that $n(X, K_X) \leq 2$, the nonvanishing is recently confirmed in \cite{Wit17} in a more general setting. However, in our case, we can verify it by a straightforward argument: Consider the nef reduction map $f: X \dashrightarrow Z$, if $n(X, K_X) = 2$, then by the adjunction formula the generic fiber $F$ of $f$ has $p_a(F) =1$, so the geometric fiber $\bar{F}$ must be smooth since char $k \geq 5$ by \cite[Prop. 2.9]{Zhang17} (or \cite[Cor. 1.8]{PW17}), which contradicts that the reduction of general fibers of $f$ are rational curves; if $n(X, K_X) = 1$, then $f$ is a morphism, and $K_X$ is the pull-back of an ample divisor on $Z$ hence is semiample.

\bigskip

Now let us discuss the proof of Theorem \ref{thm-nonvanishing}.
By \cite{Birkar16, Xu15, Zhang17}, it suffices to prove Theorem \ref{thm-nonvanishing} for $0\leq \nu(K_X) \leq 2$ and $\dim \Pic^0(X) = 0$, and by Theorem \ref{case-nc2} we can assume $\rho^*K_X  \cdot c_2(Y) \ge 0$. In this case, we closely follow the ideas in characteristic 0 firstly developed in \cite{Miy88a} (see also\cite[Secion 9]{Kol92}). However, the proof in \cite{Miy88a} used many ingredients from characteristic zero. Here we sketch our idea on how to find a positive characteristic replacement of Donaldson's theorem of stable vector bundles on smooth complex projective surfaces, and apply it to our proof: 

\medskip

To start, for separating Case (2) and (3) in the proof of  \cite[Theorem 9.0.6]{Kol92}, the finiteness of the local algebraic fundamental group for a klt singularity $x\in X$ is needed in characteristic zero for \cite[Theorem 9.0.6, Case (3)]{Kol92}. We cannot prove this in full generality in characteristic $p>0$ because of the possible wild
ramifications. Nevertheless, the argument in \cite{Xu14} implies that the tame fundamental group $\pi^{\rm loc, t}_1(x\in X)$ is always finite (see Theorem \ref{thm-finite-cover}). Therefore, in positive characteristics, we separate the cases by considering whether the tame fundamental group $\pi^{\rm  t}_1( X^{\rm sm})$ is finite.

Then in the case similar to \cite[Theorem 9.0.6, Case (2)]{Kol92}  in characteristic zero (i.e., when $\pi^{\rm  t}_1( X^{\rm sm})$ is finite), we reduce to the case that $\pi_1^{\rm et}(X^{\rm sm})$ is a pro-$p$-group. We prove Theorem \ref{nfvb} (generalizing a result of \cite[Lemma 8.1]{Lan12} which itself is built on \cite{EM10}). This is  precisely what we  use to replace Donaldson's theorem on stable vector bundles on surfaces over $\C$ (see  \cite[Page 113]{Kol92}). Then we can mimic the characteristic zero argument by applying it to the sheaf $\mathcal{E}$ obtained by an extension induced by some element in $H^2(X, nK_X)^*$  and show that $\mathcal{E}$ (up to some covers) is not strongly stable.  Eventually, by considering the destablizing exact sequence of $F_X^{e*}\mathcal{E}$, we construct sections of $nK_X$. We note that the extension involved may get split after being pulled back by some absolute Frobenius iterations, so further arguments using the construction originally from \cite{Eke88} (see \cite[Section II. 6]{Kol96}) are needed.

\bigskip

There are other smaller technical issues related to the differences between characteristic zero and characteristic $p$.

\begin{itemize}
\item
For the case $\nu(K_X) = 1$ and  when $X^{\rm sm}$ has an infinite tame fundamental group as in \cite[Proof of Theorem 9.0.6, Case (3)]{Kol92}, to follow Miyaoka's proof, we still need to deal with the failure of Hodge symmetry. Instead we will use Hodge-Witt  numbers and the symmetry of such numbers (cf. \cite[Section 6.3]{Ill82} and \cite{Eke86}).
\item
Since it is not known whether terminal singularities are Cohen-Macaulay (this is only proved in large characteristics by \cite{HW17}), dualizing complexes are used.
\item
For the case $\nu(K_X) = 2$, though Kawamata-Viehweg vanishing may fail in positive characteristic, a weaker vanishing theorem for surfaces (see \cite{Szp79, Lan09, Tan15}) is enough for our case.
\end{itemize}

\bigskip

\noindent {\bf Notation and conventions:} we follow \cite{KM98} and \cite{HL10} for the standard terminologies.

\bigskip
\noindent {\bf Acknowledgement:} We would like to thank H\'el\`ene Esnault, C.D. Hacon, Adrian Langer, Hiromu Tanaka, Liang Xiao and Weizhe Zheng for helpful conversations. We are also grateful to the referees for many valuable suggestions. Part of this work was done when the second author visited Fudan University Key Laboratory this summer, he thanks Meng Chen for his invitation. The first author was partially supported by NSFC Fund for Distinguished Young Scholars (No. 11425101). The second author was partially supported by NSFC (No. 11771260).

\section{The case $\rho^*K_X \cdot c_2(Y) <0$}
In this section, we treat the case $\rho^*K_X \cdot c_2(Y) <0$. This phenomenon only appears in positive characteristic. Recall the aimed theorem
\begin{thm}\label{nc2}
Let $X$ be a minimal projective 3-fold over an algebraically closed field $k$ with $\mathrm{char}~ k=p \geq 5$, and let $\rho: Y \to X$ be a smooth resolution of singularities. Assume that $0< \nu(X, K_X) <3$ and $\rho^*K_X \cdot c_2(Y) <0$. Then the Kodaira dimension $\kappa(X, K_X) = 1$ and $K_X$ is semiample.
\end{thm}

\begin{rem}\label{rem-unstable-T_Y}
Under the situation of Theorem \ref{nc2}, the case $n(K_X)=2$ does not happen. This case is excluded at the end of the proof, using a property of the fibration of curves, which relies on the condition $\mathrm{char}~ k \geq 5$.

We also remind that, our arguments can prove abundance under the situation that $T_Y$ is not strongly semistable w.r.t. $(D_1, D_2)$ where $D_1,D_2$ are introduced at the beginning of Sec. \ref{pf-nc2}. This is a mildly stronger result than Theorem \ref{nc2}. Indeed, by Bogomolov inequality (Theorem \ref{pos-of-ss}), the condition $\rho^*K_X \cdot c_2(Y) <0$ implies that $T_Y$ is not strongly semistable w.r.t. $(D_1, D_2)$.
\end{rem}

\subsection{Stabilities in positive characteristic}
Let $X$ be a smooth projective variety of dimension $n$ over an algebraically closed field $k$ with $\mathrm{char}~ k=p >0$, let $D_1,...,D_{n-1}$ be nef $\mathbb{R}$-Cartier divisors, and let $E$ be a torsion free coherent sheaf on $X$. The \emph{slope} of $E$ with respect to $(D_1,...,D_{n-1})$ is defined as
$$\mu(E) = \frac{c_1(E)\cdot D_1 \cdots D_{n-1}}{\rk~E}.$$
The sheaf $E$ is \emph{semistable} (resp. \emph{stable}) w.r.t. $(D_1,...,D_{n-1})$ if for every nontrivial subsheaf $E' \subset E$
$$\mu(E') \leq \mu(E)~\mathrm{(}\mathrm{resp}. <\mathrm{)}.$$
Recall that $E$ has \emph{Harder-Narasimhan filtration} (\cite[Sec. I.1.3]{HL10})
$$0 = E_0 \subset E_1 \subset E_2 \subset \cdots \subset E_m = E$$
such that the $(E_i/E_{i-1})$s are semistable and $\mu(E_i/E_{i-1}) >\mu(E_{i+1}/E_{i})$.

Denote by $F: X \cong X_1 \to X$ the Frobenius map. Recall that $F^*E$ has a canonical connection $\nabla_{can}: F^*E \to F^*E\otimes \Omega_{X_1}^1$ which is determined by
$$\mathrm{for}~ a \in \O_{X_1}~\mathrm{and}~e \in E,~~ \nabla_{can}: a\otimes e \mapsto da \otimes e,$$
and it yields the Cartier  equivalence of categories between
the category of quasi-coherent sheaves on $X$ and the category of quasi-coherent
$\mathcal{O}_{X_1}$-modules with integrable $k$-connections, whose $p$-curvature is zero (cf. \cite[Theorem 5.1]{Kat70}).

It is well known that in positive characteristic the stability is not preserved by the Frobenius pullback. We say that $E$ is \emph{strongly semistable} (resp. \emph{strongly stable}) w.r.t. $(D_1,...,D_{n-1})$ if for every integer $e\geq 0$ the pullback $F_X^{e*}E$ is semistable (resp. stable).

Let us recall some results on semistable sheaves in positive characteristic.

\begin{lem}\label{tensor}
Let $D_1,...,D_{n-1}$ be ample $\Q$-Cartier divisors. If $E_1$ and $E_2$ are $(D_1, \cdots, D_{n-1})$-semistable sheaves and $p > \rk E_1 + \rk E_2 - 2$, then $E_1\otimes E_2$
is semistable. If $E$ is semistable and $p > i(\rk E - i)$ then $\wedge^i E$ is semistable.
\end{lem}
\begin{proof}
See e.g. \cite[Corollary 1.2 and Remark 1.3]{Lan15}.
\end{proof}

\begin{thm}\label{can-conn}
Let $E$ be a semistable torsion-free
sheaf such that $F^*E$ is unstable. Let
$$0 = E_1\subset E_2 \subset \cdots \subset E_m = F^*E$$
be the Harder-Narasimhan filtration. Then the canonical connection $\nabla_{can}$ induces
non-trivial $\O_X$-homomorphisms $E_i \to (F^*E/E_i)\otimes \Omega_X^1$.
\end{thm}
\begin{proof}
See e.g. \cite[Corollary 2.4]{Lan04}.
\end{proof}

\begin{thm}[{\cite[Theorem 3.2]{Lan04}}]\label{pos-of-ss}
Let $E$ be a strongly $(D_1, ..., D_{n-1})$-semistable torsion-free
sheaf. Then
$$\Delta(E) \cdot D_2 \cdots D_{n-1} \geq 0$$
where $\Delta(E) = (\rk E-1) c_1(E)^2 - 2\rk E\cdot c_2(E)$ denotes the discriminant.
\end{thm}

\subsection{Numerical and nef dimension}
Let $X$ be a normal projective variety over an algebraically closed field $k$ and let $D$ be a nef divisor on $X$. The \emph{numerical dimension} $\nu(D)$ is defined as
$$\nu(D) = \max \{k \in \N|D^k\cdot A^{n-k} >0\mathrm{~for~an~ample~divisor}~A~\mathrm{on}~X \}.$$
It is easy to see the definition does not depend on the choice of the ample divisor $A$.

By the main result of \cite{BCE02} (see also \cite[Theorem 2.9]{CTX15}), if $k$ is uncountable, then
there exists a nonempty Zariski open set $U \subset X$ and a proper morphism $f: U \to V$, namely, the \emph{nef reduction map}, such that $D$ is numerically trivial on a very general fibre $F$ of $f$, and for a very general point $x \in U$ and $C$ a curve containing $x$, we have that $D\cdot C = 0$ if and only if $C$ is contained in the fibre of $f$.
The \emph{nef dimension} of $D$ is defined as $n(D) = \dim V$.

\subsection{$1$-foliation}
Let $X$ be a smooth variety over an algebraically closed field $k$ with $\mathrm{char}~ k=p >0$. Recall that a \emph{$1$-foliation} is a saturated subsheaf $\mathcal{F} \subset T_X$ which is involutive (i.e., $[\mathcal{F}, \mathcal{F}] \subset \mathcal{F}$) and $p$-closed (i.e., $\xi^p \in \mathcal{F}, \forall \xi \in \mathcal{F}$).
A $1$-foliation $\mathcal{F}$ induces a finite purely inseparable morphism (cf. \cite{Eke87})
$$\pi: X \to Y = X/\mathcal{F} = \mathrm{Spec} (Ann(\mathcal{F}) = _{\rm defn} \{a \in \O_X| \xi(a) = 0, \forall \xi \in \mathcal{F}\}),$$
and if $\mathcal{F}$ is locally free then $Y$ is smooth and
$$K_{X} \sim \pi^*K_{Y} + (p-1)\det \mathcal{F}|_{X}.$$

We can often get $1$-foliations when the tangent bundle is unstable by using the following criterion.
\begin{lem}\label{c-fol}
Let $\mathcal{F}$ be a saturated $\O_X$-submodule of $T_X$. If
$$\mathrm{Hom}_X(\wedge^2\mathcal{F}, T_X/\mathcal{F}) = \mathrm{Hom}_X(F_X^*\mathcal{F}, T_X /\mathcal{F}) = 0$$
then $\mathcal{F}$ is a 1-foliation.
\end{lem}
\begin{proof}
See e.g. \cite[9.1.2.2]{Kol92}, \cite[Lemma 4.2]{Eke87} or \cite[Lemma 1.5]{Lan15}.
\end{proof}

By studying the purely inseparable morphisms induced from $1$-foliations, one can prove bend-and-break type results for foliations.

\begin{thm}\textup{(\cite[Theorem 2.1]{Lan15})}\label{rat-curve}
Let $L$ be a nef $\mathbb{R}$-divisor on $X$. Let $f : C \to X$ be a non-constant morphism
from a smooth projective curve $C$ such that $X$ is smooth along $f(C)$. Let $\mathcal{F} \subseteq T_X$ be a
$1$-foliation, smooth along $f(C)$. Assume that
$$c_1(\mathcal{F})\cdot C > \frac{K_X\cdot C}{p-1}.$$
Then for every $x \in f(C)$ there is a rational curve $B_x \subseteq X$ passing through $x$ such that
$$L \cdot B_x \leq 2\dim X \frac{pL \cdot C}{(p-1)c_1(\mathcal{F})\cdot C-K_X\cdot C}.$$
\end{thm}

\subsection{Proof of Theorem \ref{nc2}}\label{pf-nc2}
Let $H$ be an ample divisor on $Y$.
We set the nef divisors
$$D_1 = \rho^*A~ \mathrm{if} ~\nu(K_X) = 1~ \mathrm{or} ~D_1 = \rho^*K_X~ \mathrm{if} ~\nu(K_X) = 2, \mbox{\ \  and \ \ } D_2 = \rho^*K_X$$
where $A$ is an ample divisor on $X$.
Then $D_1 \cdot D_2$ is nontrivial.
Denote by $\mu(\mathcal{G})$ the slope of a coherent sheaf $\mathcal{G}$ on $Y$ with respect to $(D_1,D_2)$ and by $\mu_{\epsilon}(\mathcal{G})$ the slope with respect to $(D_1 + \epsilon H,D_2 + \epsilon H)$.
Before giving the proof let us recall some well-known facts:
\begin{itemize}
\item
the Harder-Narasimhan filtration of a coherent sheaf with respect to $(D_1 + \epsilon H,D_2 + \epsilon H)$ is independent of sufficiently small $\epsilon >0$ (\cite[p. 263]{Lan04});
\item
if a sheaf $\mathcal{E}$ is $\mu_{\epsilon}$-semistable for sufficiently small $\epsilon > 0$ then it is $\mu$-semistable and
$\mu(\mathcal{E}) = \lim_{\epsilon \to 0} \mu_{\epsilon}(\mathcal{E})$; and
\item
for two torsion free sheave $\mathcal{E}_1, \mathcal{E}_2$ on $Y$, if $\mu_{\min}(\mathcal{E}_1) > \mu_{\max}(\mathcal{E}_2)$ then
$${\rm Hom}_Y(\mathcal{E}_1, \mathcal{E}_2) = 0.$$
\end{itemize}

Since $K_Y\cdot D_1\cdot D_2 = 0$ and the exceptional locus is contracted to those isolated singularities on $X$, it follows that
$$\mu(T_Y) = \mu(\Omega^1_Y) = 0~\mathrm{and}~\Delta(T_Y)\cdot \rho^*K_X = 6c_2(Y)\cdot \rho^*K_X.$$

\begin{lem}
We have $T_Y$ is not $\mu$-semistable.
\end{lem}

\begin{proof} Applying Theorem \ref{pos-of-ss}, the assumption $c_2(Y)\cdot \rho^*K_X < 0$ implies that $T_Y$ is not strongly $\mu$-semistable.

We argue by contradiction. Assume $T_Y$ is $\mu$-semistable, and for some smallest $e>0$, $F_Y^{e*}T_Y$ is not $\mu$-semistable.
Let $\mathcal{G}$ be the maximal destablizing subsheaf of $F^{e*}T_Y$ and let $\mathcal{G}' = F^{e*}T_Y/\mathcal{G}$. Then
$$\mu(\mathcal{G}) >0,~\mu(\mathcal{G}') < 0~\mbox{and}~\mu_{\max}(\mathcal{G}') < \mu(\mathcal{G}).$$
Applying Lemma \ref{can-conn}, there exists a nontrivial $\O_Y$-module homomorphism
$$\eta: \mathcal{G} \to \mathcal{G}'\otimes \Omega_Y^1.$$
By the definition of $\mathcal{G}$, it is semistable, thus we have
$$\mu(\mathcal{G}) = \mu_{\min}(\mathcal{G}) \leq \mu_{\max}(\mathcal{G}' \otimes \Omega_Y^1).$$

Case $\rk~\mathcal{G} = 2$: In this case $\rk~\mathcal{G}' = 1$, thus $\mathcal{G}'\otimes \Omega_Y^1$ is $\mu$-semistable. We get a contradiction by
\begin{eqnarray*}\label{e-dest}
\mu(\mathcal{G}) \leq \mu_{\max}(\mathcal{G}'\otimes \Omega_Y^1) = \mu(\mathcal{G}'\otimes \Omega_Y^1) = \mu(\mathcal{G}') + \mu(\Omega_Y^1) = \mu(\mathcal{G}') < \mu(\mathcal{G}).
\end{eqnarray*}
Case $\rk~\mathcal{G} = 1$:
If $\mathcal{G}'\otimes \Omega_Y^1$ is $\mu$-semistable, then as $\mu_{\max}(\mathcal{G}'\otimes \Omega_Y^1) = \mu(\mathcal{G}'\otimes \Omega_Y^1) $, we get the same contradiction as above.
Thus we may assume that $\mathcal{G}'\otimes \Omega_Y^1$ is not $\mu$-semistable, hence for sufficiently small $\epsilon >0$, $\mathcal{G}'\otimes \Omega_Y^1$ is not $(D_1 + \epsilon H, D_{2}+ \epsilon H)$-semistable. Applying Lemma \ref{tensor}, since $p\ge 5$, we have that at least one of $\mathcal{G}'$ and $\Omega_Y^1$ is not $(D_1 + \epsilon H, D_{2}+ \epsilon H)$-semistable.

If $\mathcal{G}'$ is not $(D_1 + \epsilon H, D_{2}+ \epsilon H)$-semistable, consider the Harder-Narasimhan filtration for sufficiently small $\epsilon >0$
$$0 \to \mathcal{G}'_1 \to \mathcal{G}' \to \mathcal{G}'_2 \to 0.$$
By $\mu_{\epsilon}(\mathcal{G}'_1) > \mu_{\epsilon}(\mathcal{G}'_2)$, taking limit we get
$$\mu_{\max}(\mathcal{G}') \geq \mu(\mathcal{G}'_1) \geq \mu(\mathcal{G}'_2).$$
Since $\rk~ \mathcal{G}'_i = 1$, the sheaves $\mathcal{G}'_i \otimes \Omega_Y^1$ are $\mu$-semistable, and
$$\mu(\mathcal{G}_i'\otimes \Omega_Y^1) = \mu(\mathcal{G}_i') + \mu(\Omega_Y^1) = \mu(\mathcal{G}_i') < \mu(\mathcal{G}).$$
It follows that ${\rm Hom}_Y(\mathcal{G}, \mathcal{G}'_i\otimes \Omega_Y^1) = 0$. We can conclude that ${\rm Hom}_Y(\mathcal{G}, \mathcal{G}'\otimes \Omega_Y^1) = 0$, which is a contradiction.

\medskip

For the remaining case, we assume $\mathcal{G}'$ is semistable while $\Omega_Y^1$ is not semistable with respect to $(D_1 + \epsilon H, D_{2}+ \epsilon H)$. Consider the Harder-Narasimhan filtration of $\Omega_Y^1$ with respect to $(D_1 + \epsilon H, D_{2}+ \epsilon H)$
$$0 = \mathcal{E}_0 \subset \mathcal{E}_1 \subset \cdots \subset \mathcal{E}_k = \Omega_Y^1$$
where $k=2$ or $3$. Let $\mathcal{G}_i = \mathcal{E}_i/\mathcal{E}_{i-1}$. Applying Lemma \ref{tensor}, the sheaf $\mathcal{G}_i \otimes \mathcal{G}'$ is $(D_1 + \epsilon H, D_{2}+ \epsilon H)$ (hence $(D_1, D_2)$)-semistable. Since $\Omega_Y^1$ is semistable with respect to $(D_1,D_2)$, we in fact can show that
$$\mu(\mathcal{G}_i) = \mu(\Omega_Y^1) = 0,~\mathrm{thus}~\mu_{\max}(\mathcal{G}_i \otimes \mathcal{G}') = \mu(\mathcal{G}_i \otimes \mathcal{G}')= \mu(\mathcal{G}') < \mu(\mathcal{G}).$$
Therefore ${\rm Hom}_Y(\mathcal{G}, \mathcal{G}_i \otimes \mathcal{G}') = 0$, and this implies ${\rm Hom}_Y(\mathcal{G}, \mathcal{G}'\otimes \Omega_Y^1) = 0$. We get a contradiction again.
\end{proof}

Let $\mathcal{F} \subset T_Y$ be the maximal destablizing subsheaf, and let $\mathcal{G} = T_Y/\mathcal{F}$. Then $c_1(\mathcal{F}) \cdot D_1\cdot  D_2 >0$.
\begin{lem}\label{l-foliation}
The subsheaf $\mathcal{F} \subset T_Y$ yields a $1$-foliation.
\end{lem}

\begin{proof}

If $\rk~\mathcal{F} = 1$, then $\wedge^2\mathcal{F} = 0$. And since $\mu_{\max}(\mathcal{G}) < \mu (\mathcal{F}) < \mu(F_Y^*\mathcal{F})$, we have
${\rm Hom}_Y(F_Y^*\mathcal{F}, \mathcal{G}) = 0$. It follows that $\mathcal{F}$ is a foliation by Lemma \ref{c-fol}.

\medskip

If $\rk~\mathcal{F} = 2$, then $\rk~\mathcal{G} = \rk~\wedge^2\mathcal{F} = 1$ and
$$\mu(\wedge^2\mathcal{F}) = c_1(\mathcal{F}) \cdot D_1\cdot D_2 > c_1(\mathcal{G})\cdot D_1\cdot D_2,$$
thus ${\rm Hom}_Y(\wedge^2\mathcal{F}, \mathcal{G}) = 0$. It remains to prove that ${\rm Hom}_Y(F_Y^*\mathcal{F}, \mathcal{G}) = 0$.
Since $\mu(F_Y^*\mathcal{F}) = p\cdot \mu(\mathcal{F}) > \mu(\mathcal{G})$, if $F_Y^*\mathcal{F}$ is also $\mu$-semistable, then we are done. Now we may assume that
$F_Y^*\mathcal{F}$ has a nontrivial Harder-Narasimhan filtration
$$0 \to \mathcal{F}_1 \to F_Y^*\mathcal{F} \to \mathcal{F}_2 \to 0.$$
Therefore, there exists a nonzero homomorphism
$$\mathcal{F}_1 \to \mathcal{F}_2\otimes \Omega_Y^1,$$
thus
\begin{align*}
\mu(\mathcal{F}_1) &\leq \mu_{\max}(\mathcal{F}_2\otimes \Omega_Y^1) \\
& = \mu(\mathcal{F}_2) +  \mu_{\max}(\Omega_Y^1) = \mu(\mathcal{F}_2) +  \mu(\mathcal{G}^*),
\end{align*}
i.e., $\mu(\mathcal{F}_2) \geq \mu(\mathcal{F}_1) + \mu(\mathcal{G})$.

Since $\mu(\mathcal{F}_1) > \mu(F_Y^*\mathcal{F}) = p\cdot \mu(\mathcal{F}) >0$, we also have
$\mu(\mathcal{F}_1)>\mu(\mathcal{F}_2) > \mu(\mathcal{G})$.
Therefore,
$${\rm Hom}_Y(\mathcal{F}_1, \mathcal{G}) = {\rm Hom}_Y(\mathcal{F}_2, \mathcal{G}) = 0$$
which implies ${\rm Hom}_Y(F_Y^*\mathcal{F}, \mathcal{G}) = 0$.
\end{proof}

\begin{lem}The nef dimension $$n(Y, \rho^*K_X) = n(X, K_X)\leq 2.$$
\end{lem}
\begin{proof}
Suppose that $A$ is sufficiently ample. Then for any sufficiently divisible $m$ the divisor $mK_X + A$ is also very ample (see \cite{Keeler08}) on $X$. Take a smooth curve $C_m$ as the intersection of two general divisors in $\rho^*|mK_X +A|$. We can assume $\mathcal{F}$ is smooth along $C_m$. Then
\begin{eqnarray*}
K_Y \cdot C_m& =& K_Y \cdot (m\rho^*K_X + \rho^* A)^2 \\
& =& 2mK_Y\cdot \rho^*K_X \cdot \rho^*A + K_Y\cdot (\rho^*A)^2\\
&=& 2m(K_X)^2 \cdot A + K_X\cdot A^2
\end{eqnarray*}
and
\begin{eqnarray*}
c_1(\mathcal{F}) \cdot C_m& = & c_1(\mathcal{F}) \cdot (m\rho^*K_X + \rho^* A)^2 \\
& =& m^2c_1(\mathcal{F})\cdot\rho^*K_X^2 + 2mc_1(\mathcal{F})\cdot \rho^*K_X \cdot \rho^* A + c_1(\mathcal{F})\cdot (\rho^* A)^2.
\end{eqnarray*}

For $m\gg0$, by the construction of $\mathcal{F}$, in either case $\nu(X, K_X ) = 1$ or $2$ we have
$$c_1(\mathcal{F}) \cdot C_m \gg K_Y \cdot C_m.$$
In fact, when $\nu(X,K_X)=1$, the right hand side is a constant while the left hand side increases linearly with $m$; when $\nu(X,K_X)=2$, the right hand side is a linear term while the left hand side has a quadratic term in $m$.

\medskip

Set $L = \rho^*r_0K_X$ where $r_0$ is the Cartier index of $K_X$.
 Applying Theorem \ref{rat-curve}, for general $y \in C_m$, there exists a rational curve $B_y$ such that
$$L \cdot B_y \leq 6\frac{pL \cdot C_m}{(p-1)c_1(\mathcal{F})\cdot C_m- K_Y\cdot C_m}.$$
Fix $m\gg 0$.
Since $L\cdot C_m \leq r_0K_Y\cdot C_m$, the intersection number $L \cdot B_y$ is a non-negative integer less than one, hence $L \cdot B_y = 0$.
As $C_m$ moves, we can find a family of $\mu^*K_X$-trivial curves covering $Y$, which concludes that the nef dimension is at most 2.
\end{proof}

\begin{proof}[Proof of Theorem \ref{nc2}]  Since it suffices to show the result after a base field extension, we can assume $k$ is uncountable. Consider the nef reduction map. If necessary, by blowing up $Y$ we get a fibration $f: Y \to Z$.

We first exclude the case that $Z$ is a surface. Otherwise denote by $F$ a general fiber of $f$. By the construction, $F$ does not intersect with exceptional divisors over $X$. We have that $K_Y \cdot F = \rho^*K_X \cdot F = 0$, thus $p_a(F) = 1$. Since $\mathrm{char}~ k \geq 5$, by \cite[Prop. 2.9]{Zhang17} (or \cite[Cor. 1.8]{PW17}) $F$ is a smooth elliptic curve. However, by the construction in Step 3 the fiber $F$ contains a rational curve as a component.

We have that $\nu(K_X) = n(K_X) = 1$, thus $\rho^*K_X\equiv f^*D$ for an effective divisor $D>0$ on the curve $Z$ by \cite[Lemma 5.2]{BW14}. This implies that  $\rho^*K_X\sim_{\mathbb{Q}} f^*D'$ by \cite[Theorem 1.1]{Tan15b} for some $D'\equiv D$.
Therefore, $\rho^*K_X$ is semiample, and the proof is completed.
\end{proof}

\section{Fundamental group and numerically flat vector bundles}

To mimic the proof from characteristic zero, which uses the dual relation between the fundamental groups and the flat sheaves, we need to understand the corresponding picture in positive characteristics.

\subsection{Local fundamental groups}\label{fdmgp}
In this section we will study the local fundamental groups of three dimensional klt singularities in characteristic $p>5$. In general, the local algebraic fundamental groups are finite for all klt singularities in characteristic zero (cf. \cite{Xu14, BGO17}), and for strongly $F$-regular singularities (cf. \cite{CST16}) in all dimensions. In our situation, since wild ramification may happen, using  the strategy in \cite{Xu14}, we can only prove that the local tame fundamental group  is finite in the sense of Theorem \ref{thm-finite-cover}.

\bigskip

\begin{defn}[{\cite[2.1.2]{GM71}}]\label{d-tame}
Let $L/K$ be a finite separable extension of two fields, let $v$ be a discrete valuation of $K$, and denote by $A_v$ the valuation ring, by $k(v)$ the residue field and by $\Gamma_v$ the valuation group. We say $L$ is \emph{tamely ramified} over $K$ w.r.t. $v$ if for each extension $w$ of $v$

(1)$p \nmid [\Gamma_w: \Gamma_v]$; and

(2) the residue field extension $k(w)$ is separable over $k(v)$.

Let $\pi\colon X' \to X$ be a finite morphism between normal varieties and let $D$ be a prime divisor on $X$. We say $\pi$ is tamely ramified along $D$ if the extension $K(X')/K(X)$ is tame w.r.t. $v_D$. Geometrically this means that for any divisor $D'\subset X'$ over $D\subset X$, the branch degree along $D'$ is not divisible by $p$ and the residue field extension $k(D)\subset k(D')$ is separable.

Let $X$ be a smooth variety, let $D$ be a simple normal crossing divisor, and denote by ${\rm Rev}^D(X)$ the category of the Galois covers of $X$ unramified over $X\setminus D$ and tamely ramified along $D$. Corresponding to ${\rm Rev}^D(X)$, one can define the {\it tame fundamental group} w.r.t. $D$, denoted by $\pi^{\rm t}_1(X,D)$ (see \cite[Sec.4.2,4.3]{GM71}).
\end{defn}

\medskip

\begin{lem}\label{l-pullback}
Let $\pi\colon X' \to X$ be a finite separable morphism between normal varieties over an algebraically closed field of characteristic $p>0$. Let $P$ be a codimension one point on $X$ and $D_P$ the corresponding prime divisor. Let $D'$ be the reduced divisor supported on $\pi^{-1}(D_P)$.
We have
$K_{X'}+D'-\pi^*(K_X+D)\ge 0$ and the left hand side is positive around $D'$ if and only if the branch degree along a component $D_Q$ of $D'$ over $D_P$ is divisible by $p$.
\end{lem}
\begin{proof}Let $u$ and $u'$ be uniformizers of $D_P$ and a component $D_Q$ over it, and $u\to f\cdot u'^r$ where $f$ is a unit in $\mathcal{O}_{X',Q}$. Thus
$\frac{du}{u}$ is sent to $f^{-1}df+\frac{rdu'}{u'}$. So it is a generator if and only if $r$ is not divisible by $p$. 
\end{proof}

As an immediate consequence we get the following result.
\begin{cor}\label{sing-of-cover}
Let $\pi\colon X' \to X$ be a finite morphism between normal varieties over an algebraically closed field of characteristic $p>0$. Assume that $\pi$ is unramified over the smooth locus $X^{\rm sm}$. If $X$ has at most  terminal singularities, then $X'$ is terminal.
\end{cor}
\begin{proof}We apply the calculation from characteristic zero as in \cite[5.20]{KM98}. Lemma \ref{l-pullback} implies that on the covering $X'$, the log discrepancy must be larger than or equal to the formula obtained in characteristic zero. In particular, if $X$ is terminal, then $X'$ is terminal.
\end{proof}

\begin{thm}\label{thm-finite-cover}
Let $(X,x)$ be a germ of a three dimensional algebraic klt singularity over an algebraically closed field $k$ with ${\rm char}~k = p >5$. If
\begin{align*}
\cdots \to (X_n,x_n) \xrightarrow{\eta_n} (X_{n-1}, x_{n-1}) \to \cdots \to (X_1,x_1) \xrightarrow{\eta_1} (X_0,x_0) = (X,x)
\end{align*}
is any sequence of finite covers \'{e}tale over $X \setminus \{x\}$ and tamely ramified around every divisor over $x$, then there exists some $n_0$ such that for all $n> n_0$ the covers $\eta_n$ are isomorphisms.
\end{thm}
\begin{proof}
The proof follows closely to the argument in \cite{Xu14}, but we replace the fundamental group in characteristic zero by the tame fundamental group, since the latter behaves similarly as the \'etale fundamental group in characteristic zero (see e.g. \cite{GM71}) 

By the standard construction of the Koll\'ar component (\cite[Lemma 1]{Xu14}), we can find a divisor $\Delta$ on $X$, an ample divisor $H$ on $X$ and a Koll\'ar component $E$ of $(X,\Delta)$ such that $E$ is the unique divisor such that the discrepancy $a(E, X, \Delta+ H) = -1$. Let $f\colon Y \to X$ be the morphism which only extracts the Koll\'ar component $E$. Then $(Y, E)$ is plt and $(K_{Y} + E + f_{*}^{-1}\Delta)|_{E} = K_{E} + \Gamma$ is antiample.

Let $Y_i$ be the normalization of $X_i\times_{X} Y$, and denote by $f_i\colon Y_i\to X_i$ and $\psi_i: Y_i \to Y$ the natural projections. Since $\pi_i: X_i \to X$ is tamely ramified, applying Lemma \ref{l-pullback} we conclude that an exceptional divisor $F_i$ over $X_i$ has the discrepancy $a(F_i, X_i, \pi_i^*(\Delta+ H)) = -1$ if and only if it is dominant over $E$. Let $E_i$ be one such a component, then we  define $\Gamma_i$ by $\psi_i^*(K_{Y} + E + f_{*}^{-1}\Delta)|_{E_i} = K_{E_i} + \Gamma_i$, which is just the pull back of $K_E+\Gamma$ under the finite morphism $\gamma_i\colon E_i\to E$. Since $\gamma_i$ is separable by Definition \ref{d-tame}(2),  Lemma \ref{l-pullback} implies that $({E_i},\Gamma_i)$ is klt. Thus $E_i$ yields the unique irreducible $f_i$-exceptional component and it is a Koll\'ar component of $(X_i, \Delta_i = \pi_i^*\Delta)$.

Then as in characteristic zero case,
we get a sequence of separable morphisms between klt pairs
$$\cdots \to (E_i, \Gamma_i) \to \cdots \to (E_1,\Gamma_1)\to (E_0,\Gamma_0)~\mathrm{with}~K_{E_i} + \Gamma_i = \gamma_i^*(K_{E} + \Gamma).$$
Fix an integer $N$ such that $N(K_{E}+\Gamma)$ is Cartier. Then $N(K_{E_i}+\Gamma_i)$ is Cartier. Since $(E_i,\Gamma_i)$ are all $\frac{1}{N}$-lc, by the boundedness of $\epsilon(=\frac{1}{N})$-lc log Fano surfaces \cite[Theorem 6.9]{Al94}, $\{E_i\}$ are contained in a bounded family.
As
$$\deg(\gamma_i)(-K_E-\Gamma)^2=(-K_{E_i}-\Gamma_i)^2\le \vol(-K_{E_i})$$ is bounded from above, we conclude that for sufficiently big $i$ and $j >i$, $E_j \to E_i$ is an isomorphism.

We complete the argument as in \cite[p.413]{Xu14}. By the above discussion, we may fix $i^*$ as above such that for any $j>i^*$, $E_j\to E_{i^*}$ is isomorphic. Let $\gamma$ be the generator of the tame fundamental group ($\cong \prod_{p\neq l}\mathbb{Z}_l[1]$) of $\mathcal{O}^{\rm hs}_{P_{E_{i^*}},Y_{i^*}}$ (see \cite[XIII, 5.3]{SGA1}), where $P_{E_{i^*}}$ is the generic point of $E_{i^*}$ and $\mathcal{O}^{\rm hs}_{P_{E_{i^*}},Y_{i^*}}$ is the strict henselization of the local ring $\mathcal{O}_{P_{E_{i^*}},Y_{i^*}}$. Denote by
$$\rho\colon \pi_1^{\rm loc, t}(P^{\rm hs}_{E_{i^*}}, \mathcal{O}^{\rm hs}_{P_{E_{i^*}},Y_{i^*}})\to \pi_1^{\rm loc, t}(x_{i^*}, X_{i^*})$$
the natural homomorphism. Since $E_j\to E_{i^*}$ is an isomorphism for any $j\ge i^*$, we know $Y_j\to Y_{i^*}$ is a totally ramified tame cover induced by $\pi_1^{\rm loc, t}(x_{i^*}, X_{i^*})\to \mathbb{Z}/(d_j)$ for some positive integer $d_j$ with $(d_j,p)=1$, which sends $\rho(\gamma)\in \pi_1^{\rm loc, t}(x_{i^*}, X_{i^*})$ to the generator $[1]$ of $\mathbb{Z}/(d_j)$. Therefore, it suffices to show that the image of $\gamma$ in $\pi_1^{\rm loc, t}(x_{i^*}, X_{i^*})$ is a torsion, which we will verify in the next paragraph.

In fact, to see this we can cut $Y_{i^*}$ by a general hypersurface $T_{i^*}\subset Y_{i^*}$ to get a surface $T_{i^*}$.
For the $\mathbb{Z}/(d_j)$ quotient of  $\pi_1^{\rm loc, t}(x_{i^*}, X_{i^*})$ which corresponds to $Y_j\to Y_{i^*}$, it induces a degree $d_j$ morphism $\eta: T_j:=Y_j\times_{Y_{i^*}}T_{i^*}\to T_{i^*}$ which is degree $d_j$ totally ramified around the curve $C_{i^*}:=E_{i^*}\cap T_{i^*}$. For this cyclic cover,
by considering the Galois group $(\cong \mathbb{Z}/(d_j))$-action, as ${\rm gcd}(d_j,p)=1$, we obtain a decomposition of $\eta_*\O_{T_j}$ corresponding to the eigenvalues of the Galois action,
which restricting on the smooth locus $T^{\mathrm{sm}}_{i^*}$ can be written as $\eta_*\O_{T_j}|_{T^{\mathrm{sm}}_{i^*}} \cong \bigoplus_{k=0}^{d_j-1} L^{\otimes k}$, where $L$ is a line bundle corresponding to the sections with eigenvalue being some $d_j$-th primitive root $\xi$.
Moreover, we have the relation $L^{\otimes d_j} \sim \mathcal{O}_{T^{\mathrm{sm}}_{i^*}}(-C_{i^*})$.
Thus $[C_{i^*}]=d_j \alpha_j$ for some $\alpha_j\in {\rm Cl}(T_{i^*})$ corresponding to the line bundle $L^{-1}$ on $T^{\mathrm{sm}}_{i^*}$ .
Taking the intersection with $[C_{i^*}]$ on $T_{i^*}$, we have
$$0>C_{i^*}^2=d_j(C_{i^*}\cdot \alpha_j),$$
and $[C_{i^*}]^2$ is a fixed rational number. Assume $[C_{i^*}]^2 = -\frac{s}{r}$ where $r,s$ are two positive integers prime to each other.
As $\alpha_j \in {\rm Cl}(T_{i^*})$, we know $[C_{i^*}]\cdot \alpha_j=\frac{1}{d_j}C_{i^*}^2 \in \frac{1}{e}\mathbb{Z}$, where $e$ is the Cartier index of $[C_{i^*}]$. Thus $d_j \leq es$.
\end{proof}

\subsection{Numerically flat vector bundles}

A key tool used in the proof of nonvanishing in characteristic zero is Donaldson's theorem on the equivalence between flat bundles and semistable bundles with trivial Chern classes, which is a special case of Simpson's correspondence of stable sheaves. We note that in the proof in \cite{Miy88a}, only the case when the variety has a trivial fundamental group is needed.  In this section, we recall the definition and basic results on numerically flat vector bundles. They form a neutral Tannaka category whose Tannaka dual is used to define the {\it S-fundamental group scheme} in \cite{Lan11}. Using such a theory, part of Simpson's correspondence can be recovered in characteristic $p$ (see \cite{Lan11, Lan12}). In particular,  together with the arguments in \cite{EM10}, we can obtain a similar description of numerically flat vector bundles on a variety with the fundamental group being `small', i.e., only consisting of the $p$-part. This is precisely Theorem \ref{nfvb}, which plays an analogous role as Donaldson's Theorem in Miyaoka's proof. It is a generalization of \cite[Proposition 8.2]{Lan12}, where the case $\pi_1^{\rm et}(X, x) = \{1\}$ is studied. For our purpose, we need to treat the case that $\pi_1^{\rm et}(X, x)$ is a pro-$p$-group.

\begin{defn}Let $X$ be a smooth projective variety over an algebraically closed field $k$ of characteristic $p>0$. A vector bundle $E$ on $X$ is said to be \emph{numerically flat} if both $E$ and its dual $E^*$ are nef.
\end{defn}

There are several other equivalent characterizations of numerically flat sheaves (see \cite[Theorem 2.2]{Lan12}). For example, $E$ is numerically flat if and only if $E$ is a strongly $H$-semistable torsion free sheaf for some (hence for all) ample divisor $H$ and $c_i(E) = 0$ for every $i \in \Z$.

\begin{thm}\label{nfvb}
Let $X$ be a smooth projective variety over an algebraically closed field $k$ of characteristic $p>0$ and $x\in X$ a closed point. Assume that $\pi_1^{\rm et}(X, x)$ is a pro-$p$-group. If $E$ is a numerically flat vector bundle of rank $r$ on $X$, then there exists an integer $e>0$ and an \'{e}tale cover $\pi: Y \to X$ such that $F_Y^{e*}\pi^*E \cong \oplus^r\O_Y$.
\end{thm}
\begin{proof}
First we prove the following lemma, which says that torsion points in the moduli space of stable sheaves $\mathcal{M}_r$ (\cite[Definition 3.12]{EM10}) correspond to exactly $\O_X$.
\begin{lem}\label{l1}
If $E'$ is a stable vector bundle on $X$ as in Theorem \ref{nfvb} and there exists some $e>0$ such that $F^{e*}E' \cong E'$, then $E' \cong \O_X$.
\end{lem}
\begin{proof}
Applying Lange-Stuhler's theorem \cite[Satz 1.4]{LS77}, there exists a Galois \'{e}tale cover $\eta: Z \to X$ trivializing $E'$, i.e., $\eta^*E' \cong \O_{Z}^r$. We have a $G(:=\pi_1(Z/X))$-invariant isomorphism $\eta^*E' \cong \O_{Z}^r$, and $E'\cong (\eta_*\O_Z^r)^G$, which corresponds to a representation
$$\rho: G \to {\rm GL}(r, k)$$
via the action of $G$ on $V= H^0(Z, \eta^*E') \cong k^r$.

Since $G$ is a $p$-group and $\mathrm{char}~k =p$, we have the following two immediate facts:

(1) for $g \in G$, every eigenvalue of the action of $g$ on $V$ is equal to $1$, thus the $g$-invariant subspace $V^g$ has rank $>0$;

(2) the center $H \lhd G$ is a nontrivial $p$-group, and for every $h \in H$ the $h$-invariant subspace $V^h$ is $G$-equi-invariant, i.e., $g V^h \subseteq V^h$, thus there exists a natural action of the quotient group $G/<h>$ on $V^h$.\\
By induction on the order of $G$ we can conclude that $\rk~V^{G} >0$. From this we see that $E'$ contains a subsheaf $E_0'\cong \O_X$. Therefore, if $E'$ is stable then $E'\cong \O_X$ as $\mu(E')=0$.
\end{proof}

Granted the above lemma, then we can show the following lemma by an application of the proof of \cite[Theorem 3.15]{EM10} (see also \cite[Lemma 8.1]{Lan12}). We copy them here for readers' convenience.
\begin{lem}\label{l2}
If $E'$ is a strongly stable numerically flat vector bundle on $X$ as above, then there exists $e >0$ such that $F^{e*}E'\cong \O_X$.
\end{lem}
\begin{proof}
By assumption all vector bundles $E_n = (F_X^n)^*E'$ are stable. Let $N$ be the Zariski closure of the set $\{[E_0],[E_1],...,\}$ in the moduli space $\mathcal{M}_r$. If $N$ has dimension zero then some Frobenius pullback $E_e = (F_X^e)^*E'$ is a torsion point of  $\mathcal{M}_r$. Since $\pi_1^{\rm et}(X, x)$ is a pro-$p$-group, we can apply Lemma \ref{l1} to show that $E_e = \O_X$.

Hence we can assume that $N$ has dimension at least one. Note that the set $N'$ of irreducible components of $N$ of dimension greater than or equal to 1 is Verschiebung divisible (see \cite[Definition 3.6]{EM10}),  because $V |_N$ is defined at points $E_n$ for $n \geq 1$. Let $N'_S$ be a  good model of $N'$, then the torsion points of $N'_S$ are dense in $N'_S$ by \cite[Theorem 3.14]{EM10}. Since the specialization morphism of the \'etale fundamental groups $\pi_1^{\rm et}(X) \to \pi_1^{\rm et}(X_s \times_S \bar{\mathbb{F}}_p)$ is surjective for closed $s \in S$ (see \cite[Chapter X Theorem 3.8]{SGA1}, the latter is also a pro-$p$-group, which implies that the torsion points correspond to the trivial bundle by Lemma \ref{l1}. Therefore, the trivial bundle is dense in $N'$, which is a contradiction.
\end{proof}

Recall the following well known result which in fact holds for general projective normal varieties.
\begin{lem}\label{l3}
For an extension
$0 \to \O_X \to V \to \O_X \to 0$, there exists $\eta: Z \to X$ composed of a sequence of $\mathbb{Z}/(p)$-\'{e}tale covers and Frobenius iterations of $X$ such that the pull back extension
$$0 \to \O_Z \to \eta^*V \to \O_Z \to 0$$
splits.
\end{lem}
\begin{proof}
This is proved by the well known idea of `killing cohomology'. By \cite[Proposition 12 and Section 9]{Ser58} (see also \cite[end of the proof of Lemma 1.3]{HPZ17}), the morphism $\eta$ can be chosen to be a composite of $\mathbb{Z}/(p)$-\'{e}tale covers and Frobenius morphisms.
\end{proof}

Now let us finish the proof of the theorem.  If $e$ is sufficiently large, then the Jordan-H\"{o}lder filtration
$$0 = E_0 \subset E_1 \subset E_2 \subset \cdots \subset E_r = F_X^{e*}E$$
has all the graded pieces being strongly stable (cf. \cite[Theorem 2.7]{Lan04}), and so they satisfy that $E_i/E_{i-1} \cong \O_X$ by Lemma \ref{l2}.

By induction, applying Lemma \ref{l3} we can get an \'{e}tale cover $\pi: Y \to X$ such that for $e \gg 0$,
$\pi^*F_X^{e*}E \cong F_Y^{e*}\pi^*E \cong \O_Y^r.$
\end{proof}

\section{Nonvanishing}\label{nv}
In this section we aim to prove nonvanishing Theorem \ref{thm-nonvanishing}. Abundance has been proved when $K_X$ is big in \cite{Birkar16} \cite{Xu15},  $\dim \Pic^0(X) >0$ in \cite{Zhang16, Zhang17}, or $\rho^*K_X \cdot c_2(Y) < 0$ by Theorem \ref{case-nc2}. We only need to prove
\begin{thm}\label{thm-nv}
Let $X$ be a terminal projective minimal 3-fold over an algebraically closed field $k$ of characteristic $p > 5$. Let $\rho: Y \to X$ be a smooth resolution of singularities. Assume that $\dim \Pic^0(Y) =0$, $\nu(X, K_X) <3$ and $\rho^*K_X \cdot c_2(Y) \geq 0$. Then $\kappa(X, K_X) \geq 0$.
\end{thm}

Here we note that $\dim {\rm Pic}^0(X)=\dim {\rm Pic}^0(Y)$ as the fibers of $\rho\colon Y\to X$ are covered by rational curves (see \cite[Theorem 1.1]{BW14}).

\subsection{Preliminaries} For readers' convenience, we recall some known results. The results in this section hold for any algebraically closed field of characteristic $p > 0$, i.e., we do not need $p>5$ here.

First recall the following well known result about Picard schemes.
\begin{thm}\label{pic}
Let $X$ be a normal projective variety.
Then the reduction $\Pic^0(X)_{\rm red}$ of $\Pic^0(X)$ is an abelian variety, and
$$\dim \Pic^0(X) \geq h^1(\O_X) - h^2(\O_X).$$
\end{thm}
\begin{proof}
See e.g. \cite[Remark 9.5.15, 9.5.25]{FGA}.
\end{proof}

Kawamata-Viehweg vanishing does not hold in positive characteristic, instead we will use a weaker vanishing on surfaces.
\begin{thm}\label{tavsh}
Let $X$ be a smooth projective surface and $L$ a nef and big line bundle on $X$. Then for $i>0$ and sufficiently large integer $n$
$$H^i(X, K_X + nL) = 0.$$
\end{thm}
\begin{proof}This is first proved in \cite[Proposition 2.1]{Szp79}. For more general statements, see \cite[Proposition 2.24]{Lan09} and \cite[Theorem 0.3]{Tan15}.
\end{proof}

The following result on Hodge-Witt numbers plays an analogous role as Hodge symmetry of Hodge numbers in dimension three. For a general background, see \cite[Section 6.3]{Ill82}.
\begin{thm}\label{hw}
Let $Y$ be a smooth projective variety.

(i) $h_W^{i,j}(Y) \leq h^j(\Omega_Y^i)$;

(ii) $\chi(\Omega_Y^i) = \sum_j (-1)^jh^j(\Omega_Y^i) = \sum_j (-1)^jh_W^{i,j}(Y)$ where $h_W^{i,j}(Y)$ denote the Hodge-Witt numbers of $Y$;

(iii) if $\dim Y \leq 3$ then $h_W^{i,j}(Y) = h_W^{j,i}(Y)$.
\end{thm}
\begin{proof}
The inequality (i) is called Ekedahl's inequality  (see \cite[Theorem 6.3.10]{Ill82} or \cite[p.86 Theorem 3.3]{Eke86}). The formula (ii) is known as Crew's formula (see e.g., \cite[(6.3.5)]{Ill82} or or \cite[p.85 Theorem 3.2]{Eke86}). For (iii) refer to \cite[p.113, Corollary 3.3 (iii)]{Eke86}.
\end{proof}

In the proof we come up with a non-split extension but whose pull-back via some absolute Frobenius iterations gets split. To treat this phenomenon, the idea is to use Ekedahl's construction of torsors.
\begin{thm}\label{torsor}
Let $L$ be a line bundle on a smooth projective variety $X$.
Elements of the kernel of the Frobenius action
$$F_X^*: H^1(X, L) \to H^1(X, pL)$$
give rise to $\alpha_L$-torsors. Furthermore, the torsor arising from any non-trivial element of the kernel is non-trivial, that is a purely inseparable cover $\pi': X' \to X$ of degree $p$, and
$$K_{X'} \sim \pi'^*(K_X + (p-1) L).$$
\end{thm}
\begin{proof}
The construction is due to Ekedahl \cite{Eke88}. See \cite[Sec.II.6]{Kol96} for more details.
\end{proof}

\begin{cor}\label{ndvh1}
Let $X$ be a smooth projective variety of dimension $d$, $L$ a line bundle on $X$ such that
the kernel of $F_X^*: H^1(X, L) \to H^1(X, pL)$ is nontrivial, $M$ a nef Cartier divisor and $H$ an ample divisor on $X$.
If
$$-(K_X + (p-1) L) \cdot H^{d-1} > 2d M \cdot H^{d-1}$$
then the nef dimension $n(M) \leq d-1$.
\end{cor}
\begin{proof}
Let $X'$ be the variety induced by a nonzero element of the kernel of $F_X^*: H^1(X, L) \to H^1(X, pL)$ as in Theorem \ref{torsor}. Denote by $Z \to X'$ the normalization of $X'$ and by $\pi: Z \to X$ the natural composition morphism. There exists an effective divisor $\Delta$ on $Z$, which arises from the conductor of the normalization $Z \to X'$, such that
$$K_Z + \Delta \sim_{\mathbb{Q}} \pi^*(K_X + (p-1) L).$$
For general $z \in Z$, there exists a curve $C$ arising from the intersections of $d-1$ divisors in $|k\pi^*H|$, such that $z \in C$ and that $Z$ is smooth along $C$.
Applying bend-and-break \cite[II Theorem 5.8]{Kol96}, there exists a rational curve $\Gamma_z$ passing through $z$ such that
\begin{eqnarray*}
\pi^*M\cdot \Gamma_z & \leq& 2d\frac{\pi^*M \cdot (k\pi^*H)^{d-1}}{-K_Z \cdot (k\pi^*H)^{d-1}}\\
& \leq& 2d\frac{\pi^*M \cdot (k\pi^*H)^{d-1}}{-\pi^*(K_X + (p-1) L) \cdot (k\pi^*H)^{d-1}}\\
&  \leq &\frac{2dpk^{d-1}M \cdot H^{d-1}}{-pk^{d-1}(K_X + (p-1) L) \cdot H^{d-1}} \\
& <& 1,
\end{eqnarray*}
where the ratio being smaller than $1$ is from the assumption. Since $M$ is Cartier the intersection number $\pi^*M\cdot \Gamma_z = 0$. We conclude that the nef dimension $n(M) \leq d-1$.
\end{proof}

In the same spirit, one can show the following vanishing result.
\begin{cor}\label{vanishing-of-h1}
Let $X$ be a smooth projective variety and $L$ an ample line bundle on $X$. If for every curve $C$ the inequality
$$((p-1)L -K_X)\cdot C > \dim X + 1$$
holds, then $H^1(X, -L) = 0$.
\end{cor}
\begin{proof}
This follows from \cite[II. Theorem 6.2]{Kol96}.
\end{proof}

The hard Lefschetz theorem will be used to construct global sections of a line bundle by considering its restriction on a sufficiently ample divisor.
\begin{thm}\label{thm-hl}
Let $X$ be a projective normal variety of dimension $d \geq 3$. Let $H$ be a smooth ample divisor contained in the smooth locus $X^{sm}$ of $X$ such that,
$$H^1(H, \mathcal{O}_H(-nH|_H)) = 0,~\mathrm{for~any} ~n\geq 1.$$
Then $\Pic(X) \hookrightarrow \Pic(H)$.
\end{thm}
\begin{proof}
This is \cite[XII, Corollary 3.6]{SGA2}.
\end{proof}

We will need the following simple vanishing of $H^1(E \otimes L^{-1})$ for a reflexive sheaf $E$ and a sufficiently ample line bundle $L$.
\begin{lem}\label{vanishing-of-refl}
Let $X$ be a projective normal variety of dimension with $\dim X \geq 2$. Let $E$ be a torsion free reflexive sheaf and $H$ a very ample Cartier divisor on $X$. Then $H^1(X, E(-nH))=0$ for $n\gg 0$.
\end{lem}
\begin{proof}This is well known when $E$ is a vector bundle, see \cite[III 7.8]{Har77}. For a general torsion free reflexive sheaf $E$ on a normal projective variety, we know that there exists an exact sequence
$$0\to E\to V_1\to V_2,$$
where $V_1$ and $V_2$ are vector bundles. Then we conclude by a simple diagram chase.
\end{proof}

\subsection{Proof of Theorem  \ref{thm-nv}} In this section, we aim to finish proving Theorem  \ref{thm-nv}. We can assume the smooth resolution $\rho: Y \to X$ is isomorphic over $X^{\rm sm}$ and
$$K_X\cdot \rho_*c_2(Y) \ge 0.$$

Since $\dim \Pic^0(X) = 0$ and $h^3(\O_X) = h^0(X, \omega_X)$ (\cite[p.241 Lemma 7.4]{Har77}), by Theorem \ref{pic}
\begin{eqnarray*}
\chi(\O_X) &= & h^0(\O_X) - (h^1(\O_X) - h^2(\O_X)) - h^3(\O_X) \\
 &\geq & h^0(\O_X) - h^0(X, \omega_X).
\end{eqnarray*}
We only need to consider the case $h^0(X, \omega_X) = 0$, thus we can assume $\chi(\mathcal{O}_X)>0$.
\begin{lem}
For any integer $n>0$ divisible by the Cartier index $r_0$ of $K_X$, $$\chi(X, \mathcal{O}_X(nK_X))>0.$$
\end{lem}
\begin{proof}
Applying Riemann-Roch formula, we have
\begin{align*}
&\chi(Y, \rho^*\mathcal{O}_X(nK_X)) \\
&= \frac{2n^3 - 3n^2}{12}(\rho^*K_X^3) + \frac{n}{12}(\rho^*K_X)\cdot(K_Y^2 + c_2(Y)) + \chi(\mathcal{O}_Y)\\
& = \frac{2n^3 - 3n^2}{12}K_X^3 + \frac{n}{12}K_X\cdot(K_X^2 + \rho_*c_2(Y)) + \chi(\mathcal{O}_Y) \\
& =\frac{n}{12}K_X\cdot \rho_*c_2(Y) + \chi(\mathcal{O}_Y) \\
& \geq  \chi(\mathcal{O}_Y).
\end{align*}

For $i>0$, the sheaves $R^i\rho_* \O_Y$ are supported at the union of those isolated irrational singularities on $X$.
By
\begin{align*}
\chi(Y, \rho^*\mathcal{O}_X(nK_X)) & = \chi(X, R\rho_*\O_Y \otimes\mathcal{O}_X(nK_X)) \\
& = \chi(X, \mathcal{O}_X(nK_X)) - length(R^1\rho_* \O_Y) + length(R^2\rho_* \O_Y)\\
& = \chi(X, \mathcal{O}_X(nK_X)) + \chi(\mathcal{O}_Y) - \chi(\mathcal{O}_X)
\end{align*}
we conclude that
$$\chi(X, \mathcal{O}_X(nK_X)) = (\chi(Y, \rho^*\mathcal{O}_X(nK_X)) - \chi(\mathcal{O}_Y)) + \chi(\mathcal{O}_X) \geq \chi(\mathcal{O}_X) > 0.$$
\end{proof}

If $\nu(X, K_X) = 0$, then $K_X \equiv 0$, and since $\dim \Pic^0(X) = 0$, it follows that $K_X \sim_{\Q} 0$. Thus following the proof of \cite[Theorem 9.0.6]{Kol92}, we shall divide the argument into three cases.

\bigskip

\noindent \textbf{Case (1)}: $\nu(K_X)=2$.

\medskip

Let $H$ be a smooth ample divisor on $X$ contained in the smooth locus $X^{\rm sm}$ of $X$. The exact sequence below
$$0 \to \O_X(nK_X ) \to \O_X(nK_X + H) \to \O_H(nK_X + H) \to 0,$$
induces a cohomological long exact sequence by taking cohomology
$$\to H^1(\O_H(nK_X + H)) \to H^2(\O_X(nK_X)) \to H^2(\O_X(nK_X + H)) \to\cdots.$$

Since $K_X$ is nef, applying Fujita vanishing (see \cite{Keeler03} or \cite[1.4.35-36]{Laz04}) we can assume $H$ is ample enough so that for any $n >0$ divisible by $r_0$,
$$H^2(X, \O_X(nK_X + H) )= 0.$$
Note that $K_X|_H$ is a nef big Cartier divisor. Applying Theorem \ref{tavsh} for $n_0\gg 0$ divisible by $r_0$
$$H^1(\O_H(n_0K_X + H)) \cong H^1(\O_H(K_H + (n_0 - 1)K_X|_H)) = 0.$$
Therefore $H^2(\O_X(n_0K_X)) = 0$, and in turn we have
$$h^0(\mathcal{O}_X(n_0K_X)) \geq \chi(\mathcal{O}_X(n_0K_X)) >0.$$
This finishes the proof of Case (1).

\bigskip

\begin{rem}In the following we will focus on the case $\nu(K_X)=1$.  As before, we can assume $h^2(\O_X(nK_X)) >0$.
We will split the case by considering whether $ \pi_1^{\rm t}(X^{\rm sm})$ is finite or not. Here $ \pi_1^{\rm t}(X^{\rm sm})$  is defined as the  tame fundamental group $\pi^{\rm t}_1(\overline{X},D)$ where $X^{\rm sm}\subset \overline{X}$ is a compactification with the boundary $D=\overline{X}\setminus X^{\rm}$ being simple normal crossing (see Sec. \ref{fdmgp}). For other equivalent definitions, see \cite[Theorem 1.1]{KS10}. In particular, it does not depend on the choice of the compactification $\overline{X}$.
\end{rem}

\noindent \textbf{Case (2)}: $\nu(X)=1$ and the tame fundamental group $ \pi_1^{\rm t}(X^{\rm sm})$ is finite.

\medskip

By our assumption, we can take a finite cover $\pi: X_1 \to X$ \'{e}tale over $X^{\rm sm}$ such that $\pi^{\rm t}_1(X_1^{\rm sm})$ is trivial. This implies that $\pi^{\rm et}_1(X_1^{\rm sm})$ is a pro-$p$-group. Then $\pi^*K_X = K_{X_1}$, and by Corollary \ref{sing-of-cover} $X_1$ also has terminal singularities, thus $\kappa(X_1) = \kappa(X)$. If $\dim \Pic^0(X_1) = 0$ then we have been done by \cite{Zhang17}. We can replace $X$ with $X_1$, and assume that $\pi^{\rm et}_1(X^{\rm sm})$ is a pro-$p$-group. Note that if $A$ is a smooth ample divisor on $X$ contained  in $X^{\rm sm}$, then $\pi^{\rm et}_1(A) \cong \pi^{\rm et}_1(X^{\rm sm})$ by \cite[X. Example 2.2, Theorem 3.10]{SGA2}.

To start, we fix a smooth ample divisor $H$ contained in $X^{\rm sm}$. Since $K_X$ is nef, by Fujita vanishing (\cite{Keeler03} or \cite[1.4.35-36]{Laz04}) we can assume $H$ is ample enough that $H^i(X, mK_X + H) = 0$ for $i >0$ and every integer $m >0$ divisible by $r_0$.
We fix an integer $n$ divisible by $r_0$ and such that
\begin{eqnarray}\label{e-foli}
n > \frac{(4r_0+2)K_X \cdot H^2+H^3}{K_X \cdot H^2}.
\end{eqnarray}

Denote by $\omega_X^{\bullet}$ the dualizing complex.  Applying Grothendieck duality, we get
$${\rm Ext}_{\O_X}^1(\O_X(nK_X), \omega_X^{\bullet}[-3]) \cong H^2(\O_X(nK_X))^* \neq 0.$$
Take a nonzero element $\alpha \in {\rm Ext}_{\O_X}^1(\O_X(nK_X), \omega_X^{\bullet}[-3])$. Then $\alpha$ corresponds to a non-split triangle
$$ \omega_X^{\bullet}[-3] \to \mathcal{E}^{\bullet} \to \O_X(nK_X) \to \omega_X^{\bullet}[-2].$$
Taking the $0^{\mathrm{th}}$ cohomology gives the extension
$$(*)~~~~~~0 \to \omega_X \to \mathcal{E} \to I_{T}\cdot\O_X(nK_X) \to 0,$$
where $T$ is a closed subscheme of $X$ supported at the union of those singularities of $X$ and $I_T$ is its ideal sheaf.

\begin{lem}\label{nonsplit}
The restriction of $(*)$ on $H$
\begin{eqnarray}\label{e-extension}
0 \to \O_H(K_X) \to \mathcal{F} = \mathcal{E}|_H \to \O_H(nK_X) \to 0
\end{eqnarray}
does not split.
\end{lem}
\begin{proof} Since
$${\rm Ext}_{\O_X}^1(\O_X(nK_X), \O_H(K_X)) \cong  {\rm Ext}_{\O_H}^1(\O_H(nK_X), \O_H(K_X)),$$ it suffices to prove that the restriction of $\alpha$ on $H$
$$\alpha|_H \in \mathrm{Im}({\rm Ext}_{\O_X}^1(\O_X(nK_X), \omega_X^{\bullet}[-3]) \xrightarrow{r_H} {\rm Ext}_{\O_X}^1(\O_X(nK_X), \O_H(K_X)))$$
is nonzero. Applying $R{\rm Hom}_X(\O_X(nK_X),-)$ to the following triangle
$$\omega_X^{\bullet}[-3](-H) \to\omega_X^{\bullet}[-3] \to \O_H(K_X) \to \omega_X^{\bullet}[-2](-H)$$
gives the exact sequence
{\small\begin{align*}
{\rm Ext}_{\O_X}^1(\O_X(nK_X), &\omega_X^{\bullet}[-3](-H)) \to  {\rm Ext}_{\O_X}^1(\O_X(nK_X), \omega_X^{\bullet}[-3]) \\
&\xrightarrow{r_H}  {\rm Ext}_{\O_X}^1(\O_X(nK_X), \O_H(K_X)) \to {\rm Ext}_{\O_X}^2(\O_X(nK_X), \omega_X^{\bullet}[-3](-H)).
\end{align*}}
Applying Grothendieck duality, by the construction of $H$ we have that for $i=1,2$
$${\rm Ext}_{\O_X}^i(\O_X(nK_X), \omega_X^{\bullet}[-3](-H)) \cong H^{3-i}(\O_X(nK_X + H))^* = 0.$$
In turn we conclude the restriction map $r_H$ is an isomorphism.
\end{proof}
\bigskip

We need the following weak version of Theorem \ref{nfvb} for $X$ which is not necessarily smooth.

\begin{lem}\label{lem-up-to-cover}
There exists a finite cover $\eta: Z \to X$ \'{e}tale over $X^{\rm sm}$ such that $\eta^*\mathcal{E}$ is not strongly $\eta^*H$-stable.
\end{lem}
\begin{proof}
We may assume that $\mathcal{E}$ is strongly $H$-stable.
Since $K_X^2 \cdot H = 0$, we have
$$c_1(\mathcal{E})^2\cdot H = c_2(\mathcal{E})\cdot H = \Delta(F_X^{e*}\mathcal{E}) \cdot H = 0.$$
Take a smooth divisor $A \in |kH|$ contained in $X^{\rm sm}$. Applying \cite[Theorem 5.2]{Lan04}, we can fix a sufficiently large $k$ (independent of $e$) such that $F_X^{e*}\mathcal{E}|_A$ is stable, hence
the restriction $\mathcal{F} = \mathcal{E}|_A$ is strongly $H$-stable. We note that \cite[Theorem 5.2]{Lan04} holds in the smooth setting, but in our case $X$ has isolated singularities, so we can verify the assertion on a smooth resolution, as the reflexive hull of the pull back of  $\mathcal{E}$ is strongly stable with respect to the pull back of $H$.

Let $\mathcal{G} = \mathcal{E}nd(\mathcal{F}) \cong \mathcal{F} \otimes \mathcal{F}^*$. Then $\mathcal{G}$ is strongly semistable by Lemma \ref{tensor}. We have
$$c_1(\mathcal{G})=0 ~\mathrm{and} ~c_2(\mathcal{G}) = 4c_2(\mathcal{F}) - c_1(\mathcal{F})^2 = 0.$$
The sheaf $\mathcal{G}$ is numerically flat by \cite[Theorem 2.2]{Lan12}. Since $\pi^{\rm et}_1(A)$ is a pro-$p$-group, applying Theorem \ref{nfvb}, we can find an \'{e}tale cover $\eta_A: A' \to A$ and $e>0$ such that
$$F_{A'}^{e*} \eta_A^*\mathcal{G} \cong \eta_A^*F_A^{e*} \mathcal{G} \cong \oplus^4\O_{A'}.$$
It follows that
$$\dim {\rm Hom}_{A'}(F_{A'}^{e*} \eta_A^*\mathcal{F}, F_{A'}^{e*} \eta_A^*\mathcal{F}) = h^0(\mathcal{E}nd(F_{A'}^{e*} \eta_A^*\mathcal{F})) = 4.$$

By $\pi^{\rm et}_1(A) \cong \pi^{\rm et}_1(X^{\rm sm})$, the cover $\eta_A \colon A' \to A$ extends to a finite cover $\eta: Z \to X$ \'{e}tale over $X^{\rm sm}$,
where $Z$ is assumed to be normal. We will show that $\eta^*\mathcal{E}$ is not strongly $\eta^*H$-stable.
Since stable sheaves do not admit nontrivial endomorphisms, it suffices to prove that the sheaf $\tilde{\mathcal{G}}^e =_{\rm defn} \mathcal{E}nd_Z(F_Z^{e*}\eta^*\mathcal{E})$ satisfies
$h^0(Z, \tilde{\mathcal{G}}^e )>1.$

\medskip

We fix an integer $m$ such that $m(p-2)H -K_X$ is ample, and at the beginning we can choose $k > m + 4$, then
$$(p-1)A|_A - K_A = (p-1)A|_A - (K_X +A)|_A = ((p-2)kH -K_X)|_A$$
satisfies the condition in Corollary \ref{vanishing-of-h1}. Obviously this property holds for $A'$ on $Z$. Applying Corollary \ref{vanishing-of-h1} gives
$$H^1(A', \O_{A'}(-lA')) = 0~\mathrm{for\  every}~l>0.$$
Then by $\tilde{\mathcal{G}}^e|_{A'} (-lA') \cong \oplus^4\O_{A'}(-lA')$, it follows immediately that
$$H^1(A', \tilde{\mathcal{G}}^e|_{A'} (-lA')) = 0~\mathrm{for \ every}~l>0.$$
Consider the long exact sequence induced by taking cohomology of the short exact sequence
$$0 \to \tilde{\mathcal{G}}^e|_{A'} (-lA') \to \tilde{\mathcal{G}}^e|_{(l+1)A'}  \to \tilde{\mathcal{G}}^e|_{lA'}\to 0.$$
We can show the restriction map
$$H^0(A', \tilde{\mathcal{G}}^e|_{(l+1)A'}) \to H^0(A', \tilde{\mathcal{G}}^e|_{lA'})$$ is surjective.
By induction on $l$, we conclude $h^0(\tilde{\mathcal{G}}^e|_{lA'}) \geq 4$ for every $l$.

Consider the exact sequence
$$0 \to \tilde{\mathcal{G}}^e(-lA') \to \tilde{\mathcal{G}}^e \to \tilde{\mathcal{G}}^e|_{lA'} \to 0.$$
Then for a sufficiently large $l$, Lemma \ref{vanishing-of-refl} tells that $H^1(Z, \tilde{\mathcal{G}}^e(-lA')) = 0$.
In turn we conclude that $\dim H^0(Z, \tilde{\mathcal{G}}^e) \ge 4$.
\end{proof}

\bigskip

Denote the pull back $H'=\eta^{-1}(H)$ and $\eta_H\colon H'\to H$. We claim that the pull-back of the extension (\ref{e-extension}) via $\eta_H$
\begin{eqnarray}\label{e-extension-pullback}
0 \to \O_{H'}(\eta^*K_X) \to \eta_H^*\mathcal{F} = \eta^*\mathcal{E}|_{H'} \to \O_{H'}(\eta^*nK_X) \to 0
\end{eqnarray}
does not split. For this we only need to show that the pull-back homomorphism
$$\gamma: H^1(H, \O_H((1-n)K_X)) \to H^1(H', \eta_H^*\O_H((1-n)K_X))$$
is injective. Up to the isomorphism below
$$H^1(H', \eta_H^*\O_H((1-n)K_X)) \cong H^1(H, \eta_{H*}\O_{H'} \otimes \O_H((1-n)K_X))$$
the map $\gamma$ fits into the following  exact sequence
$$H^0(H, \mathcal{C} \otimes \O_H((1-n)K_X)) \to  H^1(H, \O_H((1-n)K_X)) \to H^1(H, \eta_{H*}\O_{H'} \otimes \O_H((1-n)K_X))$$
which is induced by tensoring $\mathcal{O}((1-n)K_X)$ with the natural exact sequence
$$0 \to \O_H \to \eta_{H*}\O_{H'} \to \mathcal{C} \to 0$$
and taking the cohomology.
To show the injectivity of $\gamma$, we only need to prove the vanishing
$$H^0(H, \mathcal{C} \otimes \O_H((1-n)K_X)) = 0.$$
Denote by $F_H^e: H_e \to H$ and $F_{H'}^e: H'_e \to H'$ the $e^{\rm th}$ absolute Frobenius iterations. Since $\eta_{H}: H' \to H$ is \'{e}tale, we have the following commutative diagram
$$\xymatrix@C=1.5cm{&H'\cong H'_e\ar[dr]^{\eta_H}\ar@/^2pc/[rr]|{F_{H'}^{e}}\ar[r]^{\cong} &H'\times_H H_e\ar[d]\ar[r] &H'\ar[d]^{\eta_H}\\
& &H\cong H_e \ar[r]^{F_H^e} &H
}$$
Then $F_H^{e*} (\eta_{H*}\O_{H'}) \cong  \eta_{H*}(F_{H'}^{e*}\O_{H'}) \cong \eta_{H*}\O_{H'}$ by \cite[III.9.3]{Har77}, and thus $F_H^{e*}\mathcal{C} \cong \mathcal{C}$.
Since $\mathcal{C}$ is torsion free and $\nu(H, K_X|_H) = 1$, if $e\gg 0$ then
$$H^0(H, F_{H}^{e*}(\mathcal{C} \otimes \O_H((1-n)K_X))) \cong H^0(H, \mathcal{C} \otimes \O_H(p^e(1-n)K_X)) = 0.$$
We can conclude $H^0(H, \mathcal{C} \otimes \O_H((1-n)K_X)) = 0$ since the map below is injective
$$F_H^{e*}: H^0(H, \mathcal{C} \otimes \O_H((1-n)K_X)) \to H^0(H, F_{H}^{e*}(\mathcal{C} \otimes \O_H((1-n)K_X))).$$


\bigskip

Let us proceed with the proof. In the following we will replace $X, H$ with $Z, H'$ introduced in the lemma above, and replace the extension $(*)$ with its pull-back via $\eta$. We note that the inequality \eqref{e-foli} still holds on $Z$ and $H'$.

By Lemma \ref{lem-up-to-cover}, we may assume that $F_X^{e*}\mathcal{E}$ is not $H$-stable for some $e>0$. There exists an exact sequence
$$0 \to \mathcal{A} \to F_X^{e*}\mathcal{E} \to \mathcal{B} \to 0$$
$\mu$-destabilizing $F_X^{e*}\mathcal{E}$. We may write that
$$\mathcal{A}^{**} = \O_X(A)~ \mathrm{and}~ \mathcal{B}^{**} = \O_X(B)$$
where $A, B$ are Weil divisors on $X$.
Then we get nonzero homomorphisms
$$\eta_1: \O_X(A) \to \O_X(p^enK_X)~\mathrm{and}~\eta_2: \O_X(p^eK_X) \to \O_X(B).$$
In particular both $p^enK_X - A$ and $B-p^eK_X$ are linearly equivalent to some effective divisors.

\begin{lem}\label{equiv-restr}
Let $S$ be a smooth ample divisor on $X$ contained in $X^{\rm sm}$. Then we have  $\Q \cdot A|_S = \Q \cdot K_X|_S$ in ${\rm NS}(S)\otimes \Q$.
\end{lem}
\begin{proof}This follows from the same argument as in \cite[p.113-114]{Kol92}. For readers' convenience, we give all the details. Put
$$A|_S = a, B|_S = b, S|_S = s~\mathrm{and}~K_X|_S = k.$$
Then
$$a+b = p^e(n+1)k,~s\cdot (a-b) \geq 0~\mathrm{and}~a\cdot b \leq c_2(F_X^{e*}\mathcal{E}|_S) = 0.$$

We claim that $(a- b)^2 = 0$. Otherwise,
$$(a-b)^2 = (a+b)^2 - 4a\cdot b = - 4a\cdot b>0,$$
applying \cite[ V.1.8]{Har77} shows that $a-b \in C^{+}(S)$ (i.e., the positive cone of $S$), so that $h^0(\O_S(m(a-b))) = O(m^2)$ for $m\gg 0$. Since $b-p^ek$ and $np^ek-a$ are both effective and
$$mp^e(n-1)k = m(a-b) + m(b-p^ek) + m(np^ek-a),$$
it follows that
$$h^0(\O_S(mp^e(n-1)k)) \geq h^0(\O_S(m(a-b))) = O(m^2)~\mathrm{for}~m\gg 0$$
which contradicts that $\nu(S, k) = 1$.

Since $\O_S(a) \hookrightarrow \O_S(p^enk)$ and $\O_S(p^ek) \hookrightarrow \O_S(b)$, we get $\O_S(a-b) \hookrightarrow \O_S(p^e(n-1)k)$. Since $k^2=0$ and $k$ is nef, we find that $k\cdot(a-b) = 0$. Since $k$ and $(a-b)$ both lie in the closure of $C^{+}(S)$, the Hodge index theorem gives $a-b \in \Q\cdot  k$ in $NS(S)\otimes \Q$, thus $a,b\in \Q\cdot  k$.
\end{proof}

If $\dim{\rm Pic}^0(S)>0$, we can always find an $l(\neq p)$-torsion in ${\rm Pic}(S)$, which will induces a $\mathbb{Z}/(l)$ quotient of $\pi_1^{\rm et}(S)$. Thus it follows from that $\pi_1^{\rm et}(S)$ is a pro-$p$-group, ${\rm Pic}^0(S)$ is of dimension zero. We can assume $S$ is sufficiently ample to satisfy Lefschetz Hyperplane Theorem \ref{thm-hl} $\Pic(X) \hookrightarrow \Pic(S)$, indeed by
$$(p-1)nS - K_S = (p-1)nS - (K_X + S)|_S = (((p-1)n -1)S - K_X)|_S,$$
for $n>0$ if $S$ is sufficiently ample then $H^1(S, -nS|_S) = 0$ by Corollary \ref{vanishing-of-h1}.
From Lemma \ref{equiv-restr} we conclude that there exist positive integers $n_1,n_2$ such that $n_1A \sim n_2K_X$.
The inclusion $\eta_1: \O_X(A) \hookrightarrow \O_X(p^enK_X)$ induces another inclusion
$$\O_X(n_1A) \cong \O_X(n_2K_X) \hookrightarrow \O_X(n_1p^enK_X).$$
We obtain that $H^0(X, (n_1p^en - n_2)K_X) \neq 0$.

\bigskip

If the pull-back of the extension \eqref{e-extension}
\begin{eqnarray}\label{e-extension2}
0 \to F_H^{e*}\O_H(K_X) \to F_H^{e*}\mathcal{E} \to F_H^{e*}\O_H(nK_X) \to 0
\end{eqnarray}
does not split, then $\O_H(A|_H) \hookrightarrow \O_H(p^enK_X|_H)$ is not an isomorphism, thus $n_1p^en - n_2 \neq 0$. The theorem follows easily.

\medskip

The above argument applies for every smooth divisor of $|H|$ contained in $X^{\rm sm}$. The remaining case follows from the lemma below.

\begin{lem}
We assume  for some $$n > \frac{(4r_0+2)K_X \cdot H^2+H^3}{K_X \cdot H^2}$$
and general smooth divisors in $ |H|$,
the extension \eqref{e-extension2} splits. Then $K_X$ is semiample.
\end{lem}
\begin{proof}
We have a minimal natural number $1\leq a \leq e$ such that the class
$$\alpha \in \ker(F_H^{a*}: H^1(\O_H((1 - n)K_X)) \to H^1(F_H^{a*}\O_H((1 - n)K_X))),$$
while $F_H^{(a-1)*}\alpha$ is nonzero in $H^1(F_H^{(a-1)*}\O_H((1 - n)K_X))$. Set
$$L =p^{a-1}(1 - n)K_X|_H \mbox{\ \  and \ \ }M = r_0K_X|_H,$$ then by our assumption on $n$,
we know
$$-(K_H+(p-1)L)\cdot H > 4M\cdot H.$$

Applying Corollary \ref{ndvh1} shows
$n(H, K_X|_H) \leq 1$. As this holds for general members in $|H|$, we conclude $n(X, K_X) \leq 2$, and then we can proceed to prove the claim by the same argument as in Theorem \ref{nc2}.
\end{proof}

\bigskip

\noindent   \textbf{Case (3):} $\nu(X)=1$ and $ \pi_1^{\rm t}(X^{\rm sm})$ is infinite.

\medskip

By assumption we have an infinite tower of nontrivial quasi-\'{e}tale Galois covers $$\cdots \to X_n \to X_{n-1} \to \cdots \to X_1 \to X_0 = X,$$
which are all tamely ramified over $X^{\rm sm}$.
By Theorem \ref{thm-finite-cover}, there exists $n_0$ such that for $n>n_0$, the covers $X_{n+1} \to X_n$ are all \'{e}tale. Therefore, we can assume the resolutions $Y_{n+1}\to Y_n$ are also finite \'etale for $n>n_0$. Similarly as in the previous case, we have that $X_n$ has terminal singularities and $\kappa(X) = \kappa(X_n)$, and if $\dim \Pic^0(X_n) = 0$ then we have been done by \cite{Zhang17}. So we can replace $X$ with $X_n$, and always assume $\dim \Pic^0(X) = 0$, thus $\chi(\O_Y) >0$ by the same argument as at the beginning of the proof. Up to a finite \'{e}tale cover we can assume that $\chi(\O_Y) \geq 4$.
Applying the properties of Hodge-Witt numbers summarized in Theorem \ref{hw}, we know
$$\chi(\O_{Y}) = \sum_j (-1)^jh_W^{0,j}(Y) =  \sum_j (-1)^jh_W^{j,0}(Y) \leq h_W^{0,0}(Y) +  h_W^{2,0}(Y),$$
which implies
$$h^0(\Omega_{Y}^2) \geq h_W^{2,0}(Y) \geq 3.$$ The linearly independent global sections of $\Omega_Y^2$ give a $\O_Y$-homomorphism
$\O_Y^{\oplus 3} \to \Omega_Y^2$. Take the saturation $\mathcal{E} \subset \Omega_Y^2$ of the image of this homomorphism.

\medskip

If $\rk ~\mathcal{E} = 3$, then $\Omega_Y^2$ is generically globally generated, thus
$$H^0(Y, 2K_Y) \cong H^0(Y, \det \Omega_Y^2) \neq 0.$$

If  $\rk ~\mathcal{E} = 1$ or $2$, then $\dim |\det \mathcal{E}| \geq 1$. We can write that $|\det \mathcal{E}| = |M| + F$ where $|M|$ denotes the movable part and $F$ denotes the fixed part.
Consider the sheaf
$$\mathcal{F} = \mathcal{E} \otimes \omega_Y^{-1} \subset \Omega_Y^2\otimes \omega_Y^{-1} \cong T_Y.$$
Let $H$ be an ample divisor on $X$. Then
\begin{eqnarray*}
c_1(\mathcal{F}) \cdot \rho^*K_X \cdot \rho^*H &= & c_1(\mathcal{E})\cdot \rho^*K_X \cdot \rho^*H + (\rk~\mathcal{E})(-K_Y) \cdot \rho^*K_X \cdot \rho^*H \\
&=& c_1(\mathcal{E})\cdot \rho^*K_X \cdot \rho^*H \\
& \geq & 0.
\end{eqnarray*}

If $c_1(\mathcal{E})\cdot \rho^*K_X \cdot \rho^*H >0$, then $T_Y$ is unstable with respect to $(\rho^*K_X, \rho^*H)$.
So abundance holds in this case by Remark \ref{rem-unstable-T_Y}.

If $c_1(\mathcal{E})\cdot \rho^*K_X \cdot \rho^*H = 0$ then $M\cdot \rho^*K_X \cdot \rho^*H= 0$. Assume $H$ is a smooth ample divisor contained in the smooth locus $X^{\rm sm}$.
By the Hodge index theorem we have
$$\Q\rho_*M|_H = \Q K_X|_H \mbox{ in } {\rm NS}(H)\otimes \Q,$$ and by the above argument we can assume this holds for all sufficiently ample $H$. We conclude that $\Q\rho_*M  = \Q K_X$ in ${\rm NS}(X)\otimes \Q$. There exists some $t \in \Q$ such that $\rho_*M  \equiv tK_X$, which concludes $\rho_*M \sim_{\Q} tK_X$ since $\dim \Pic^0(X) = 0$.

\bigskip

In conclusion we complete the proof of all cases.


\end{document}